\documentclass[12pt,a4paper]{article}

\usepackage{amsmath}
\usepackage{amsthm}
\usepackage{amsfonts}
\usepackage{amscd}
\usepackage{amssymb}
\usepackage{mathrsfs}
\usepackage[all,cmtip]{xy}
\newtheorem{thm}{Theorem}[section]
\newtheorem*{thm*}{Theorem}

\newtheorem*{cor*}{Corollary}

\newtheorem*{lem*}{Lemma}
\newtheorem{prop}[thm]{Proposition}

\theoremstyle{definition}

\newtheorem*{defn*}{Definition}

\theoremstyle{remark}

\newcommand{\Q}{\mathbb Q}
\newcommand{\R}{\mathbb R}
\newcommand{\C}{\mathbb C}

\newcommand{\Z}{\mathbb Z}
\newcommand{\PP}{\mathbb P}

\newcommand{\K}{\mathscr K}

\newcommand{\HH}{\mathfrak H}

\newcommand{\I}{\mathrm i}

\DeclareMathOperator{\Tr}{Tr}

\DeclareMathOperator{\kernel}{Ker}

\DeclareMathOperator{\sign}{sign}
\renewcommand{\Im}{\mathrm{Im}\;}
\renewcommand{\Re}{\mathrm{Re}\;}

\def\bbfZ{\mathbb Z}

\author{Anton Mellit \footnote{mellit@gmail.com}}

\title{Elliptic dilogarithms and parallel lines}

\begin{document}

\maketitle

\begin{abstract}
We prove Boyd's conjectures relating Mahler's measures and values of L-functions of elliptic curves in the cases when the corresponding elliptic curve has conductor $14$. 
\end{abstract}

\section{Introduction}
\subsection{Parallel lines}
Boyd's conjectures are identities of the form $m(P)=a\pi^{-2}L(E,2)$, $a\in \Q^\times$, where $P$ is a polynomial in two variables and $E$ is an elliptic curve over $\Q$ (definition of the Mahler measure $m$ follows). By results of \cite{Deninger97}, \cite{Beilinson85}, \cite{Bloch00}, \cite{SS88}, \cite{RodriguezVillegas99} both sides of the conjectured identities are reduced to relations between values of the elliptic dilogarithm. To prove relations between elliptic dilogarithms one usually tries to construct rational functions $f$ such that divisors of both $f$ and $1-f$ are supported on a given set of points.

Let $E/\C$ be an elliptic curve and $Z\subset E(\C)$ be a finite subgroup. Let us realize $E$ as a plane cubic with equation $y^2=x^3+ax+b$ for $a,b\in\C$. For each triple $p,q,r\in Z\setminus \{0\}$ such that $p+q+r=0$ consider the line $l_{p,q,r}$ passing through $p$, $q$, $r$. Let $s_{p,q,r}$ be the slope of this line. Suppose $s_{p,q,r}=s_{p',q',r'}$ for another triple of points, which is equivalent to the lines $l_{p,q,r}$ and $l_{p',q',r'}$ being parallel. Then from equations of these lines one can obtain two functions $f$, $g$ on $E$ such that $f+g=1$ and divisors of $f$ and $g$ are supported on $Z$. Thus we obtain (hopefully a non-trivial) relation between values of the elliptic dilogarithm at points of $Z$.

We propose to search for parallel lines as above in two ways. The first way, dubbed ``breadth-first search'', is to fix $Z = \Z/m \times \Z/m'$ and consider the moduli space of elliptic curves $E$ with embedding $Z\rightarrow E$. Then for any two triples $p,q,r$ and $p',q',r'$ the difference $s_{p,q,r}-s_{p',q',r'}$ is a modular form of weight $1$ on the moduli space, in fact an Eisenstein series, and at the points where the form is zero we obtain a relation. We demonstrate that using this approach, in particular, one can obtain relations from \cite{Bertin04}.

Another approach, which I call ``depth-first search'', is to fix a curve $E$ and consider some big subgroup $Z$ hoping that when $Z$ is big enough some parallel lines will appear. However, this seems to work only for some ``nice'' curves. One of the cases when this does work is presented in this paper.

\subsection{Boyd's conjectures}
Rogers provided a table of relations between Mahler's measures and values of L-functions of elliptic curves of low conductors $11$, $14$, $15$, $20$, $24$, $27$, $32$, $36$ in \cite{Rogers08}. Among these relations some had been proved and some had not. According to Rogers, those relations which involve curves with complex multiplication (conductors $27$, $32$, $36$) were all proved. Except those, only a relation with curve of conductor $11$ was proved. Let us list the relations with curves of conductor $14$.

Let $P\in \C[y,z]$. The Mahler measure of $P$ is defined as
\[
M(P) := (2\pi\I)^{-2} \int_{|y|=|z|=1} \log |P(y,z)| \frac{dy}{y} \; \frac{dz}{z}.
\] 
Denote
\begin{align*}
n(k)&:=M(y^3 + z^3 + 1 - k y z),\\
g(k)&:=M((1+y)(1+z)(y+z) - k y z).\\
\end{align*}

Let $E_{14}$ be the elliptic curve of conductor $14$ with Weierstrass form $y^2 + y x + y = x^3 + 4 x - 6$. It is isomorphic to the modular curve $X^0(14)$ with the pullback of the N\'eron differential $\frac{dx}{2 y + x + 1}$ given by the following eta-product \cite{MO97} 
\[
 f_{14}:=\eta(\tau) \eta(2 \tau) \eta(7 \tau) \eta(14 \tau).
\]

Then $L(E_{14},s)=L(f_{14},s)$, and the relations listed by Rogers are
\begin{align}
 n(-1) &= \frac{7}{\pi^2} L(f_{14},2),\label{eq:conjectures1}\\
 n(5) &= \frac{49}{2\pi^2} L(f_{14},2),\label{eq:conjectures2}\\
 g(1) &= \frac{7}{2\pi^2} L(f_{14},2),\label{eq:conjectures3}\\
 g(7) &= \frac{21}{\pi^2} L(f_{14},2),\label{eq:conjectures4}\\
 g(-8) &= \frac{35}{\pi^2} L(f_{14},2).\label{eq:conjectures5}\\
\end{align}

\subsection{The regulator}\label{the_regulator}
Fix a smooth projective curve $C/\C$. Consider an element $\xi=\sum_i\{f_i,g_i\}\in \Lambda^2 \C(C)^\times$. If the tame symbol of $\xi$ vanishes at every point of $C$, it defines an element in $K_2(C)$. Then its regulator $r_C(\xi)$ is an element of $H^1(C,\R)$ whose value on $[\gamma]\in H_1(C,\Z)$ is 
\[
r_C(\xi)([\gamma]) = \int_\gamma \sum_i log|f_i| d \arg g_i - \log|g_i| d \arg f_i.
\]

Using Poincar\'e duality one can also evaluate the regulator on forms. Namely, let $\omega$ be a holomorphic $1$-form on $C$. The value of the regulator on $\omega$ is defined as follows:
\begin{equation}\label{cup_reg}
\langle r_C(\xi), \omega\rangle := \langle r_C(\xi) \cap \omega, [C]\rangle = 2 \int_C \sum_i \log|f_i| d\arg g_i \wedge \omega.
\end{equation}

Denote by $\K_n$ (resp. $\K_g$) the set of values of the function $\frac{y^3+z^3+1}{yz}$ (resp. $\frac{(1+y)(1+z)(y+z)}{yz}$) on the torus $|y|=|z|=1$. Then, according to Deninger (\cite{Deninger97}, \cite{RodriguezVillegas99}),
for $k\notin \K_n$ (resp. $k\notin \K_g$) one can express $n(k)$ (resp. $g(k)$) as $\frac{1}{2\pi} r_C(\{y,z\})([\gamma])$
for a certain $[\gamma]\in H_1(C,\Z)$, where $C$ is the projective closure of the equation $y^3+z^3+1-k yz=0$ (resp. $(1+y)(1+z)(y+z)-k y z=0$). When $k$ is on the boundary of $\K_n$ (resp. $\K_g$) Deninger's result still applies by continuity.

\subsection{Elliptic dilogarithm}
Now let $C=E$ be an elliptic curve over $\C$. Let $\Z[E(\C)]^-$ be the group of divisors on $E$ modulo divisors of the form $[p] + [-p]$ for $p\in E(\C)$. Define a map from $\Lambda^2 \C(E)^\times$ to $\Z[E(\C)]^-$ by
\[
 \beta: \{f,g\} \rightarrow (f)\ast (g)^-
\]
where ``$\ast$'' and ``$-$'' mean the convolution and the antipode operations on the group of divisors of $E$.
Fix an isomorphism $E\cong\C/\langle 1,\tau \rangle$ for $\tau\in\HH$. Let $u$ be the coordinate on $\C$. Let $x\in E(\C)$, $x=a\tau+b$ for $a,b\in\R$. Put 
\[
 R(\tau,x) = -\frac{\I}{\pi}(\Im{\tau})^2\sum_{(m,n)\in \Z\times\Z \setminus{( 0,0)}}\frac{\sin(2\pi(na-mb))}{(m\tau+n)^2(m\overline\tau+n)},
\]
As in \cite{Zagier90} (it seems that the sign there is wrong), the real part of $R$ is the elliptic dilogarithm $D(\tau,x)$. By a result of Bloch, we have
\[
 \langle r_E(\{f,g\}), du\rangle = R(\tau,(f)\ast (g)^-).
\]
For a holomorphic $1$-form $\omega$ on $E$ put
\[
R_{E,\omega}(x)=\frac{\omega}{du} R(\tau,x).
\]
Then it is trivial to verify that $R_{E,\omega}$ does not depend on the choice of the isomorphism $E\cong\C/\langle 1,\tau \rangle$.

By linearity we extend $R_{E,\omega}$ to $\Z[E(\C)]^-$.

The function $R_{E,\omega}(x)$ satisfies the following properties:
\begin{enumerate}
 \item For any $\lambda\in\C$ $R_{E,\lambda\omega}(x) = \lambda R_{E,\omega}(x)$.
 \item For an isogeny $\varphi:E'\rightarrow E$ and $x\in E(\C)$
\begin{equation}\label{eq:dilog:isog}
 R_{E,\omega}(x) = \sum_{x'\in\varphi^{-1}(x)} R_{E',\varphi^* \omega}(x').
\end{equation}
 \item For a function $f\in \C(E)^\times$, $f\neq 1$, one has $R_{E,\omega}((f)\ast(1-f)^-) = 0$.
\end{enumerate}

We expect that any algebraic relation between $R_{E,\omega}(x)$, where $E,\omega,x$ are defined over $\overline{\Q}$, follows from the relations just listed.

\subsection{Elliptic curves over $\R$}
Let $E$ be an elliptic curve defined over $\R$. Then there is a unique up to sign isomorphism $E\cong\C/\langle 1,\tau \rangle$ which is compatible with the complex conjugation. Then $\tau$ can be chosen so that either $\tau\in\I\R$ or $\tau\in \frac12 + \I\R$. We will write $R_E(x)$ for $R_{E,du}(x) = R(\tau,x)$ (the elliptic dilogarithm).

Let $\xi$ be such that $[\xi]\in K_2(E)$ and suppose $\xi$ is defined over $\R$. Let $\gamma_1$ and $\gamma_2$ be the loops corresponding to the segments $[0,1]$ and $[0,\tau]$ respectively. It can be seen that $r_C(\xi)([\gamma_1])=0$. Therefore we can evaluate
\[
\langle r_E(\xi), du \rangle = r_E(\xi)([\gamma_1]) \int_{\gamma_2} du - r_E(\xi)([\gamma_2]) \int_{\gamma_1} du = -r_E(\xi)([\gamma_2]).
\]
Thus, in general, for the loop $\gamma$ corresponding to the path $[0,a\tau + b]$ we obtain
\[
r_E(\xi)([\gamma]) = -b R_E(\beta(\xi)).
\] 

\subsection{Remarks}

We mention that statements essentially equivalent to some cases of the distribution relation (\ref{eq:dilog:isog}) appeared in \cite{KO05}, \cite{LR07}, \cite{Touafek08}.

In the rest of the paper we prove the conjectures (\ref{eq:conjectures1}-\ref{eq:conjectures5}) and fill some details to what was stated above.

A natural idea (see \cite{GL98}, Lemma 3.29) is to search not only for pairs of parallel lines, but for triples of incident lines passing through a given finite set of points on an elliptic curve. Computer searches performed by the author and F. Brunault indicate that the cases of Boyd's conjectures for conductors $20$, $24$ can be proved by this method.

Using a certain argument communicated to the author by R. de Jeu (unpublished) it is possible to prove that relations in $K_2$ of an elliptic curve coming from all triples of incident lines span all relations. However, this includes lines defined over all number fields passing through all (not necessarily torsion) points. So it is not clear how this may help to study $K_2$ of, say, arbitrary elliptic curve over $\Q$.

% We also define an element $R_E(x)\in H^2(E,\R)$ for each $x\in E(\C)$ such that
% \[
%  \langle R_E(x) \cap \omega, [E]\rangle = R_{E,\omega}(x).
% \]
% It follows that $r_E(\{f,g\}) = R_E(\beta(\{f,g\}))$ for any $\{f,g\}\in K_2(E)$.
\subsection{Acknowledgements}
The author would like to thank W. Zudilin for bringing his attention to the problem. He is also grateful to M. Rogers, F. Brunault, H. Gangl, A. Levin, F. R. Villegas, A. Goncharov and D. Zagier for interesting discussions and to the Max Planck Institute for Mathematics in Bonn for its hospitality and stimulating environment.

\section{Families of elliptic curves}

\subsection{General remarks}
First we note a few facts about the elliptic curves $P(y,z)=0$ where $P$ is a polynomial $y^3+z^3+1-kyz$ or $(1+y)(1+z)(y+z)-kyz$ (these cases are called $n$ and $g$ respectively).
\begin{enumerate}
\item The line at infinity of the projective closure of the affine plane with coordinates $y$, $z$ has a natural choice of homogeneous coordinates, namely $(y:z)$. The point $(1:-1)$ always belongs to $C$, we choose this point as zero for the elliptic curve.
\item The map $(y, z) \rightarrow (z, y)$ is the map $p\rightarrow -p$.
\item In the case $g$ the point $Q=(0,0)$ belongs to the curve. We have $2Q=0$.
\item In the cases $g$ the map $(y,z)\rightarrow(y^{-1}, z^{-1})$ preserves the equation, sends $0$ to $Q$ and does not have fixed points on the curve. Therefore it is the map $p\rightarrow (p+Q)$.
\end{enumerate}

We would like to know when the curve is singular
\begin{prop} The elliptic curve becomes singular in the following cases:
\begin{itemize} 
\item For the family $n$ if $k=3$ or $k^2 + 3k + 9=0$.
\item For the family $g$ if $k\in\{-1,0,8\}$.
\end{itemize}
\end{prop}

Next we determine the number $c_\R(k)$ of connected components of the set of real points for $k\in\R$. This is half the number of real $2$-torsion points. Taking remark (ii) into account, we need to find whether there are $1$ or $3$ real solutions of $P(y, y)=0$. We obtain
\begin{prop} $c_\R(k)=1$ in the following cases:
\begin{itemize} 
\item For the family $n$ if $k<3$.
\item For the family $g$ if $0<k<8$.
\end{itemize}
Otherwise $c_\R(k)=2$.
\end{prop}

And finally, we describe $\K_\ast$.
\begin{prop}
\begin{itemize} 
\item $\K_n$ is a curved triangle with vertices at the solutions of $x^3=27$. Its intersection with the real line is $[-1,3]$;
\item $\K_g=[-1, 8]$;
\end{itemize}
\end{prop}

\subsection{Determining $[\gamma]$}\label{determine_gamma}
In general our curve $E(k)$ is determined by a polynomial equation $P_k(y, z) = 0$. To find $\gamma$ from Section \ref{the_regulator} we need to represent the real torus $\{(y,z)|\; |y|=|z|=1\}$ as the boundary of some $3$-chain $\xi$ which intersects $C$ transversally, then $\gamma=\xi\cap C$. For $\xi$ we choose the set $\{(y,z)|\; |y|=|z|\geq 1\}$. The orientation of $\xi$ is encoded by a $1$-form normal to $\xi$, i.e. a form on $\PP^2$ which does not vanish on $\xi$ but vanishes on vectors along $\xi$. If $e_1$, $e_2$, $e_3$ is an oriented basis of tangent vectors to $\xi$ at some point and $e_4$ is such that the value of the form is positive on $e_4$, the tuple $e_1$, $e_2$, $e_3$, $e_4$ is required to be an oriented basis for the ambient space. One may verify that $\rho:=\Re(\frac{dy}y-\frac{dz}z)$ can be chosen as such a form. 

We will always have $E(k)$ defined over $\R$. We choose an orientation on $E(k)(\R)$ by requiring the form $\omega$ to be positive. The form $\omega$ is defined simultaneously for the whole family as the unique form which satisfies
\[
\omega\wedge dk = \frac{dy}{y}\wedge \frac{dz}{z}.
\]
One can find $\omega$ by differentiating the equation $P_k(y,z)=0$ (note that for all our families $\frac{\partial}{\partial k} P_k = -yz$):
\[
\left(\frac{\partial}{\partial y} P_k\right) dy + \left(\frac{\partial}{\partial z} P_k\right) dz - yz dk=0.
\]
Thus
\[
\frac{dy}{y}\wedge \frac{dz}{z} = \frac{dy}{\left(\frac{\partial}{\partial z} P_k\right)} \wedge dk,
\]
so
\[
\omega = \frac{dy}{\left(\frac{\partial}{\partial z} P_k\right)}.
\]
Next we express $\rho$ in terms of $\omega$:
\[
\rho = \Re\left(\frac{\Big(\frac{\partial}{\partial z} P_k\Big)}{y}\omega + \frac{\Big(\frac{\partial}{\partial y} P_k\Big)}{z}\omega\right) = \Re\left(\frac{\left(y \frac{\partial}{\partial y} + z\frac{\partial}{\partial z}\right) P_k}{yz} \cdot \omega \right) = \Re (f_k(y,z)\omega)
\]
where 
\[
f_k(y,z):=\frac{\left(y \frac{\partial}{\partial y} + z\frac{\partial}{\partial z}\right) P_k}{yz}.
\]
Suppose the class of $\gamma$ is the class of $[0,t_k-c_k\tau]$. Then, because $\Re\omega$ is a normal form to paths of type $[0,-x\I]$, $x>0$, we obtain (provided that $f_k$ has constant sign on $|\gamma|$)
\[
\sign c_k = \sign (f_k|_{|\gamma|}).
\]

To determine $[\gamma]$ we compute the intersection $\xi\cap C$ first set-theoretically, and then restrict $f_k$ to the intersection. Thus we recover the orientation of $\gamma$, hence $c_k$ and 
\[
r_{E(k)}(\{y,z\})([\gamma]) = c_k R_{E(k)}(\beta(\{y,z\})).
\]

\subsubsection{The case $n$}
Assume $k\notin \K_n$, $k\in\R$, which means $k<-1$ or $k>3$. We make substitutions $y=ta$, $z=tb$, and look for the solutions of
\begin{equation}\label{eqgEn}
t^3 a^3 + t^3 b^3 + 1 = kt^2 ab,\quad |a|=|b| = 1, \quad t\geq 1.
\end{equation}
Applying the complex conjugation and multiplying by $a^3b^3$ we obtain 
\[
t^3 b^3 + t^3 a^3 + a^3 b^3 = k t^2 a^2 b^2.
\]
Subtracting these two equations we obtain
\[
1- (ab)^3 = k t^2 a b (1-ab).
\]
We see that either $ab=1$, or $1 + ab + (ab)^2 = kt^2 a b$, which can be rewritten as
\[
ab + (ab)^{-1} = k t^2 - 1.
\]
Now $ab + (ab)^{-1} \in [-2, 2]$. On the other hand, when $k>3$, then $k t^2 - 1 >2$. When $k<-1$, then $k t^2 - 1<-2$, thus the last equality cannot hold and we conclude $ab=1$. This is equivalent to $z=\overline{y}$, which can be rewritten as $p=-\overline{p}$. The points which satisfy the latter condition will be called ``imaginary points''.

Now we prove that $|\gamma|$ is not empty. Taking into account that $ab=1$ we need to show that there exist solutions of (see (\ref{eqgEn}))
\[
t^3(a^3+a^{-3}) + 1 = k t^2
\]
with $t\geq 1$, $|a|=1$. Thus we need to show that there exists $t\geq 1$ such that $\frac{k t^2 -1}{t^3}\in[-2,2]$. This is obviously satisfied when $t$ is large enough.

Note also that for $k>3$ not all imaginary points have the property $|y|\geq 1$. Indeed, for $t=\frac{1}{\sqrt{k}}<1$ we also have $\frac{k t^2 -1}{t^3}=0\in[-2,2]$.

Assume $k<-1$. Then the set of real points has one connected component, so $\tau\in \frac12 + \I\R$ and the set of imaginary points also has one connected component. Since $|\gamma|$ is not empty it must coincide with the set of imaginary points, so it is the segment $[0, 2\tau - 1]$ in the uniformisation. 

Assume $k>3$. Then the set of real points has two connected components, so $\tau\in \I\R$ and the set of imaginary points also has two connected components. Since $|\gamma|$ is not empty and not the whole set of imaginary points, it must be one of the connected components. It is also clear that it is the one that contains the zero point, so $|\gamma|=[0,\tau]$.

To determine the orientation of $\gamma$ we have to look at the function $f_k(y,z)=\frac{kyz-3}{yz}$. We have $yz\in\R$, $yz>1$, so when $k>3$ it is positive and when $k<-1$ it is negative. The case $k=-1$ can be obtained by passing to the limit. Therefore
\[
c_k=\begin{cases}
 1 & \text{if $k>3$,}\\
 -2 & \text{if $k\leq -1$.}
\end{cases}
\]

\subsubsection{The case $g$}
Similarly to the previous case, if $k\in\R$ (here we do not require $k\notin\K_g$):
\begin{equation}\label{eqgEg}
ab(a+b) t^2 + ((a+b)^2-kab)t + a + b = 0.
\end{equation}
Applying the complex conjugation and multiplying by $a^3b^3$ we obtain 
\[
ab(a+b) t^2 + ab ((a+b)^2-kab)t + (ab)^2 (a + b) = 0.
\]
Again, subtracting the last two equations we obtain
\[
(1-ab)(((a+b)^2-kab)t + (1+ab)(a+b)) = 0.
\]
Again, $ab=1$ or $((a+b)^2-kab)t + (1+ab)(a+b)=0$. In the latter case we obtain (using (\ref{eqgEg}))
\[
ab(a+b)(t^2-1)=0.
\]
Thus $a+b=0$ or $t=1$. $a+b=0$ gives $kabt=0$, hence $k=0$. If $t=1$ then $k = (a+b)(1+a)(1+b)$, which is possible only when $k\in[-1,8]$. Note that $t=1$ means $\overline{y}=y^{-1}$ and $\overline{z}=z^{-1}$. Thus it is equivalent to $\overline{p}=p+Q$.

Thus for $k\notin[-1,8]$ we again have $p=-\overline{p}$. For $k\in[-1,8]$ we have two possibilities: $p=-\overline{p}$ or $\overline{p}=p+Q$. Next we note that $|\gamma|$ and $|\gamma| + Q$ together cover the set of imaginary points because the map $p\rightarrow p+Q$ is the map $(y,z)\rightarrow(y^{-1}, z^{-1})$. Moreover, each point which does not satisfy $p=p+Q$ is covered exactly once. When $k\in(0,8)$ we have $\tau\in\frac12 + \I\R$ and the set of imaginary points is the segment $[0, 2\tau-1]$. We see that $\gamma$ must be homologous to the path $[0, a\tau+b]$ with $|a|=1$. Otherwise $\tau\in\I\R$ and the set of imaginary points has two connected components each homologous to $[0,\tau]$. Thus, again $\gamma$ must be homologous to the path $[0, a\tau+b]$ with $|a|=1$. 

To determine the orientation we put $ab=1$ in (\ref{eqgEg}) to obtain
\[
t (a+b)^2 + (t^2+1) (a+b) = k t.
\]
We need to determine the sign of the function 
\[
f_k(y,z)=\frac{3 yz(y+z) + 2 (y+z)^2 + (y+z)}{yz} - 2k
\]
on the locus $\{|y|=|z|>1\}$. Using the equation of the curve and the variables $t$, $a$, $b$ (we have $ab=1$) we may rewrite it as $(a+b)(t^2-1)$, whose sign is the same as the sign of $a+b$. Taking into account that $a+b\in[-2,2]$ we see that the sign of $a+b$ coincides with the sign of $t (a+b)^2 + (t^2+1) (a+b)$, whose sign coincides with the sign of $k$. Thus $c_k=\sign k$.

\subsection{Deuring form}
For $a_1, a_3\in\C$ we denote by $D(a_1, a_3)$ the elliptic curve given by the Deuring equation
\[
 Y^2 + a_1 X Y + a_3 Y = X^3.
\]
We denote by $P$ the point $(0,0)$ of order $3$ and by $\omega$ the differential form $\frac{dX}{2 Y + a_1 X + a_3}$. When $a_1, a_3$ are real, the real period of $\omega$ is denoted $\Omega_\R(\omega)$.

Now we reduce the curves corresponding to the Mahler measures $n$ and $g$ to the Deuring forms and express the Mahler measures as elliptic dilogarithms. We require that the form $\omega$ from Section \ref{determine_gamma} in the Deuring form becomes $\frac{dX}{2 Y + a_1 X + a_3}$, so the signs do not change.

\subsubsection{Case $n$}\label{subs:case_n}
Let $C$ be the projective curve defined by the projective closure of the equation $y^3+z^3+1=k y z$. It is a smooth curve when $k^3\neq 27$. Making the following rational change of variables:
\[
X=-(k^2+3k+9)\frac{1+y+z}{k+3(y + z)},\quad Y=(k^2+3k+9)\frac{ky+3z+3}{k+3(y + z)},
\]
one obtains the curve $D(k+6, k^2+3k+9)$:
\begin{equation}\tag{$E_n$}\label{eq:n}
 Y^2+(k+6)XY+(k^2+3k+9)Y=X^3.
\end{equation}

The curve $C$ contains $9$ points among which three satisfy $y=0$, three satisfy $z=0$, and three points are at infinity. Let $\epsilon=\frac{-1+\sqrt{-3}}2$. The point $(0,-1)$ on $C$ is mapped to the point $P=(0,0)$, and the point at infinity $(1:-\epsilon:0)$ is mapped to $Q=(-\frac{k^2+3k+9}3,\frac{k^2+3k+9}9(\epsilon+2)(k-3\epsilon))$. The points $P$ and $Q$ are generators of the group of points of order $3$ thus giving a level structure of type $\Gamma(3)$ on the elliptic curve.

One can verify that the divisor of $y$ is $[P]+[P+Q]+[P-Q]-[0]-[Q]-[-Q]$ and the divisor of $z$ is $[-P]+[-P+Q]+[-P-Q]-[0]-[Q]-[-Q]$. Therefore
\begin{multline*}
\beta(\{y,z\})=(y)\ast(z)^- = 3([2P]+[2P+Q]+[2P-Q]) - 6([P]+[P+Q]+[P-Q]) \\+ 3([0]+[Q]+[-Q]) = -9 ([P]+[P+Q]+[P-Q]).
\end{multline*}

Finally we express the Mahler measure in terms of $R_{E_n(k)}$:
\begin{equation}\label{eq:deninger_n}
 n(k) = -\frac{9 c_k}{2\pi} R_{E_n(k)}([P] + [P+Q] + [P-Q]),\;\;\text{where}\;
c_k=\begin{cases}
 1 & \text{if $k>3$,}\\
 -2 & \text{if $k\leq -1$.}
\end{cases}
\end{equation}

\subsubsection{Case $g$}\label{subs:case_g}
Similarly to the first case after making the following rational change of variables:
\[
X=k\frac{y+z+1}{y + z-k},\quad Y=k\frac{-yk+z+1}{y+z-k},
\]
one obtains the elliptic curve $D(k-2,k)$ if $k\notin\{-1,0,8\}$:
\begin{equation}\tag{$E_g$}\label{eq:g}
 Y^2+(k-2)XY+kY=X^3.
\end{equation}

Let $Q$ be the point $(-1,-1)$, which is of order $2$. The points $P$ and $Q$ generate a cyclic group of order $6$. Thus we obtain a level structure of type $\Gamma_1(6)$. 

One can verify that
\[
(y) = [Q]+[P]-[Q+P]-[0],\quad (z) = [Q]+[-P]-[Q-P]-[0],
\]
therefore
\[
\beta(\{y,z\})=(y)\ast(z)^- = 6 ([P+Q]-[P]).
\]

Finally we obtain
\begin{equation}\label{eq:deninger_g}
 g(k) = \frac{3 c_k}{\pi} R_{E_g(k)}([P+Q]-[P]),\;\;\text{where}\;
c_k=\sign k.
\end{equation}

\section{Beilinson's theorem for $\Gamma_0(N)$}\label{sec:beilinson}
Let $N$ be a squarefree integer with prime decomposition $N=p_1,\ldots,p_n$. Let $f=\sum a(n)q^n$ be a newform for $\Gamma_0(N)$ of weight $2$. Let $W$ be the group of Atkin-Lehner involutions. This is a group isomorphic to $(\Z/2\Z)^n$. For $m>0$, $m|N$ denote by $w_m$ the Atkin-Lehner involution corresponding to $m$. Any cusp of $\Gamma_0(N)$ is given by $w(\infty)$ for a unique $w\in W$. The width of $w_m(\infty)$ is $m$. It is known that (see \cite{AL70}) for a prime $p|N$ we have $f|_2w_p=-a(p) f$. 

Let $E(\tau, s)$ be the real-analytic Eisenstein series which for $\Re s>1$ is given by
\[
 E(\tau, s) = \sum_{\gamma\in \Gamma_0(\infty) \backslash \Gamma_0(N)} (\Im \gamma\tau)^s.
\]
Let $\Q[W]_0$ be the augmentation ideal of $\Q[W]$. For any $\alpha\in \Q[W]_0$, $\alpha = \sum_{w\in W} \alpha_w [w]$ consider $E_\alpha = E|_0 \alpha$.
By the Kronecker limit formula $E_\alpha(\tau,s)$ is holomorphic by $s$ at $s=1$ and $E_\alpha(\tau,1)=-\frac{1}{2\pi} \log|F_\alpha(\tau)|$ for $F_\alpha\in \C(X_0(N))^\times \otimes \Q$ such that $(F_\alpha) = \sum_{w\in W} \alpha_w [w(\infty)]$.

Let $\alpha, \beta\in \Q[W]_0$. Then $\{F_\alpha, F_\beta\} \in K_2(X_0(N))\otimes \Q$ and by the definition
\begin{multline}\label{eq:beilinson:1}
\langle r_{X_0(N)}(\{F_\alpha, F_\beta\}), 2\pi\I f(\tau) d\tau\rangle \\= - 8 \pi^2 \int_{X_0(N)} f(\tau) \log|F_\alpha(\tau)| \overline{\frac{\partial}{2\pi\I \partial \tau} \log F_\beta(\tau)} d x d y
\end{multline}
where $x$, $y$ denote $\Re \tau$, $\Im \tau$ respectively. 

Let $E_2=\frac{3}{\pi y} + 1-24 \sum_{k,l>0} k q^{kl}$. Note that $\frac{\partial}{2\pi\I \partial \tau} \log F_\beta(\tau)$ is a holomorphic Eisenstein series of weight $2$. Suppose $\beta'\in\Q[W]$ is such that the latter equals $E_2|_2 \beta'$ (we will compute $\beta'$ later). Then the right hand side of (\ref{eq:beilinson:1}) equals
\[
 \lim_{s\rightarrow 1+} 16 \pi^3 \int_{X_0(N)} \sum_{w\in W} \alpha_w E(w \tau, s) f(\tau) \overline{(E_2|_2\beta')(\tau)} dx dy.
\]

Let $\gamma:W\rightarrow\{\pm 1\}$ be such that $f|_2 w = \gamma(w) f$ for all $w\in W$. We may rewrite the integral as 
\[
\int_{X_0(N)} E(\tau,s) f(\tau) \overline{(E_2|_2 \beta' \sum_{w\in W} \alpha_w \gamma(w) w)(\tau)} dx dy.
\]
Let $g=E_2|_2 \beta' \sum_{w\in W} \alpha_w \gamma(w) w$. The last integral is the Rankin-Selberg integral for the convolution L-function of $f$ and $g$. It equals (see \cite{Iwaniec97}) $(4 \pi)^{-2} (L(f) \ast \overline{L(g)})(s+1)$ where 
\[
\sum_{k>0} a(k) k^{-s} \ast \sum_{k>0} a'(k) k^{-s} := \sum_{k>0} a(k)a'(k) k^{-s}.
\]
Therefore our regulator (\ref{eq:beilinson:1}) equals $-\frac12 (L(f) \ast \overline{L(g)})(2)$.

Let $\beta' \sum_{w\in W} \alpha_w \gamma(w) w = \sum_{m|N} c_m w_m$. Note that $E_2|_2 w_m (\tau) = m E_2(m \tau)$. Thus
\[
 L(g)(s) = - 24 \sum_{m|N} m c_m \sum_{k,l>0} k (klm)^{-s} = -24 \sum_{m|N} c_m m^{1-s} \zeta(s)\zeta(s-1).
\]

Now we compute $L(f) \ast m^{-s} \zeta(s)\zeta(s-1)$. Since both $L(f)$ and $m^{-s} \zeta(s)\zeta(s-1)$ have product expansions it is enough to find the convolution of local factors. For $p\nmid N$ we get
\[
 \frac{1}{1-a(p) T + p T^2} \ast \frac{1}{(1-T)(1-pT)} = \frac{1-p^2 T^2}{(1-a(p) T + p T^2) (1 - p a(p)T + p^3 T^2)}.
\]
For $p|\frac{N}{m}$:
\[
 \frac{1}{1-a(p) T} \ast \frac{1}{(1-T)(1-pT)} = \frac{1}{(1-a(p) T) (1-p a(p) T)}.
\]
For $p|m$:
\[
 \frac{1}{1-a(p) T} \ast \frac{T}{(1-T)(1-pT)} = \frac{a(p) T}{(1-a(p) T) (1-p a(p) T)}.
\]
Summarizing:
\[
 L(f) \ast m^{-s} \zeta(s)\zeta(s-1) = a(m) m^{-s} \frac{L(f, s) L(f, s-1)}{\zeta(2s-2) \prod_{k=1}^n (1-p_k^{2-2s})}. 
\]
Thus the regulator equals
\[
 -\frac{144}{\pi} \frac{\sum_{m|N} c_m a(m) m^{-1}}{\prod_{k=1}^n (1-p_k^{-2})} L(f,1) L(f,2).
\]

We want to obtain a nicer expression for $\sum_{m|N} c_m a(m) m^{-1}$. Let $\varepsilon:\Q[W]\rightarrow \Q$ be the map which sends the identity to $1$ and all other $w\in W$ to $0$. Then
\begin{equation}\label{eq:beilinson:2}
 \sum_{m|N} c_m a(m) m^{-1} = \varepsilon(\beta' (\sum_{w\in W} \alpha_w \gamma(w) w) (\sum_{m|N} a(m) m^{-1}w_m)).
\end{equation}
To find $\beta'$ we note the the constant coefficient of $E_2|_2 w_m$ is $m$ for all $m|N$. Let $d = \sum_{m|N} m w_m$. Then $\varepsilon(w_m d) = m$ for all $m|N$. On the other hand the constant coefficient of $E_2|_2 w \beta'$ must be $\beta_w$. This is satisfied if $\varepsilon(w \beta' d) = \beta_w$ for all $w\in W$, which is equivalent to $\beta' d = \beta$. Since $d=\prod_{k=1}^n (1 + p_k w_{p_k})$, $d$ is invertible, $\beta'= d^{-1}\beta$, and 
\[
 d^{-1} = \prod_{k=1}^n \frac{1 - p_k w_{p_k}}{1-p_k^2}.
\]
For the last term in (\ref{eq:beilinson:2}) we also have a product expansion, namely 
\[
 \sum_{m|N} a(m) m^{-1}w_m = \prod_{k=1}^n (1 - \gamma(w_{p_k})p_k^{-1} w_{p_k}) = a(N) N^{-1} w_N \prod_{k=1}^n (1 - \gamma(w_{p_k})p_k w_{p_k}).
\]

Let $\gamma^*$ be the involution of $\Q[W]$ which sends $w$ to $\gamma(w) w$ for $w\in W$. Put $\alpha'=d^{-1} \alpha$. Then (\ref{eq:beilinson:2}) becomes
\[
 (-1)^n \gamma(w_N) N^{-1} \varepsilon(w_N \beta' \gamma^*(\alpha')) \prod_{k=1}^n (1-p_k^2) = N^{-1} \varepsilon(w_N \alpha' \gamma^*(\beta')) \prod_{k=1}^n (p_k^2-1).
\]

The final formula for the regulator reads
\begin{equation}\label{eq:beilinson:3}
 \langle r_{X_0(N)}(\{F_\alpha, F_\beta\}), 2\pi\I f(\tau) d\tau\rangle = -\frac{144 N}{\pi} \varepsilon(w_N \alpha' \gamma^*(\beta')) L(f,1) L(f,2).
\end{equation}

Note that corresponding formula in a different case was obtained by Brunault \cite{Brunault06}.

\section{Search for parallel lines}\label{sec:search}
Let $(E/\C, \omega)$ be an elliptic curve with a holomorphic differential. Realize $E$ as a plane cubic with equation $y^2=x^3+a x + b$ such that $\omega = \frac{dx}{2y}$. Let $Z$ be a finite subgroup of $E(\C)$. Put $Z_0=E\setminus\{0\}$. For any $p\in Z_0$ let $x_p$, $y_p$ be the coordinates of $p$. Let $T$ be the set of (unordered) triples $p,q,r \in Z_0$ such that $p+q+r=0$. For $(p,q,r)\in T$ let $l_{p,q,r}$ be the unique line on $P^2$ such that $l_{p,q,r}\cdot E=[p]+[q]+[r]$. Let the equation of this line be $y + s_{p,q,r} x + t_{p,q,r}=0$. The slopes $s_{p,q,r}$ are of special interest because of the following observation:
\begin{prop}\label{prop:parallel}
 Suppose $(p,q,r)\in T$ and $(p',q',r')\in T$ be two distinct triples such that $s_{p,q,r}=s_{p',q',r'}$. Then the value of $R_\omega$ is zero on the divisor
\[
 ([p]+[q]+[r]-3[0])\ast([-p']+[-q']+[-r']-3[0]).
\]
\end{prop}
\begin{proof}
 Since the triples are distinct $t_{p,q,r}\neq t_{p',q',r'}$. Put 
\[
f = (t_{p,q,r}-t_{p',q',r'})^{-1} (y + s_{p,q,r} x + t_{p,q,r}).
\]
Then $(f) = [p]+[q]+[r]-3[0]$, $(1-f) = [p']+[q']+[r']-3[0]$ and the statement is essentially $R_\omega(\beta(\{f,1-f\}))=0$.
\end{proof}

\begin{prop}\label{prop:z} There exists a unique map from $Z_0$ to $\C$, denoted $p\rightarrow z_p\in\C$, such that 
\begin{enumerate}
 \item[(i)] $z_p+z_{-p}=0$ for all $p\in Z_0$,
 \item[(ii)] $z_p + z_q + z_r = s_{p,q,r}$ for all $(p,q,r)\in T$,
\end{enumerate}
moreover, we have
\begin{enumerate}
 \item[(iii)] $x_p + x_q + x_r = s_{p,q,r}^2$ for all $(p,q,r)\in T$.
\end{enumerate}
For an isogeny $\rho:E'\rightarrow E$ let the differential on $E'$ be $\rho^* \omega$ and the group be $\rho^{-1} Z$. Then for any $p\in Z_0$
\[
 z_p = \sum_{p'\in \rho^{-1} p} z_{p'}.
\]
\end{prop}
\begin{proof}
 {\em Existence}. As a local coordinate at $0$ choose the formal integral of $\omega$ and denote it by $u$. For each $p\in Z_0$ let $n_p$ be the order of $p$ and $f$ be a function such that $(f)=n_p[0]-n_p[p]$, which is unique up to a constant multiple. Put $\varphi_p=n_p^{-1} d log f$. Let $e_k(p)$ be the coefficients of $\varphi_p$:
\[
 \varphi_p = \sum_{k=0}^\infty e_k(p) u^{k-1} du.
\]
Since $\varphi_{-p}$ is the pullback of $\varphi_p$ under the automorphism  $\alpha\rightarrow -\alpha$ of $E$, we have $e_k(-p)=(-1)^k e_k(p)$. Note that
\[
 - d \log (y + s_{p,q,r} x + t_{p,q,r}) = \varphi_p + \varphi_q + \varphi_r.
\]
Therefore the Laurent coefficients of the left hand side are sums of the values of $e_k$. On the other hand we may use the Laurent expansions of $x$ and $y$ to obtain
\[
 - d \log (y + s_{p,q,r} x + t_{p,q,r}) = (3 u^{-1} + s_{p,q,r} + s_{p,q,r}^2 u + \cdots) du.
\]
Thus $z_p=e_1(p)$ satisfies the desired properties. 

To prove $(iii)$ it suffices to see that $e_2(p)=x_p$. This follows from vanishing of the sum of residues of the differential $x \varphi_p$ \footnote{The same method can be applied to compute $e_k$ for $k>2$}.

{\em Uniqueness}. We show how to determine $z_p$ for a given $p\in Z_0$ using the properties $(i)$ and $(ii)$. For $k=1,\cdots, n_p-2$ we have $z_{(k+1)p} = z_p + z_{kp} - s_{p,kp,-(k+1)p}$. Summing up these identities we obtain
\[
 n_p z_p = \sum_{k=1}^{n_p-2} s_{p,kp,-(k+1)p}.
\]

To prove the last assertion note that the isogeny preserves the local parameter and 
\[
 \rho^* \varphi_p = \sum_{p'\in \rho^{-1} p} \varphi_{p'}.
\]
\end{proof}

\subsection{Multiplication by $2$}
Let $p\in Z_0$ and suppose all points $q$ such that $2q=p$ belong to $Z$. Denote these points $q_1$,\ldots,$q_4$. Consider the isogeny $[2]:E\rightarrow E$. The pullback of $\omega$ is $2\omega$. The following relation for $R_{E,\omega}$ follows from (\ref{eq:dilog:isog}):
\begin{equation}\label{eq:dilog:mult2}
 R_{E,\omega}(p) = 2 \sum_{i=1}^4 R_{E,\omega}(q_i).
\end{equation}
By Proposition \ref{prop:z} we have (note that passing from $2 \omega$ to $\omega$ divides $z_{q_i}$ by $2$)
\[
 z_p = \frac{1}{2} \sum_{i=1}^4 z_{q_i}.
\]

Suppose $p$ is a point of order $3$. Then $2p$ is among $q_i$. Let it be $q_4$. Hence
\[
 3 z_p = z_{q_1}+z_{q_2}+z_{q_3}.
\]
Since $3p=q_1+q_2+q_3=0$, we may consider lines $l_{p,p,p}$ and $l_{q_1,q_2,q_3}$. By Proposition \ref{prop:z} their slopes are equal. The corresponding divisor from Proposition \ref{prop:parallel} is 
\[
 3([p]-[0])\ast ([-q_1]+[-q_3]+[-q_4]-3[0]) \sim_{-} 3(2[q_2]+2[q_3]+2[q_4] - 3 [p]),
\]
and the statement of Proposition \ref{prop:parallel} is equivalent to (\ref{eq:dilog:mult2}).

\begin{prop}
 Let $p$ be a point of order $3$. Let $r_1$, $r_2$, $r_3$ be the points of order $2$. Then the triple tangent at $p$ is parallel to the line passing through the points $2p + r_i$.
\end{prop}

\subsection{Lines for $\Gamma(3)$-structure}\label{subs:lines:3}
Let $Z$ be the group of points of order $3$. The set $T$ consists of triples $(p,p,p)$. For each $p\in Z_0$ we have $s^2=3x$. Thus the only pair of parallel lines can be $l_{p,p,p}$ and $l_{-p,-p,-p}$, and this happens only if $x_p=0$. The slope must be $0$, so $3x_p^2 + a=0$, which implies $a=0$. Thus our curve is isomorphic to the CM curve $y^2=x^3+1$. For $p=(0,1)$ the lines $l_{p,p,p}$ and $l_{-p,-p,-p}$ are parallel and Proposition \ref{prop:parallel} says $R_{E,\omega}(p)=0$, which also could be proved by (\ref{eq:dilog:isog}) for the automorphism of $E$ of order $3$, which preserves $p$.

The corresponding rational values of $k$ in the family (\ref{eq:n}), for which parallel lines exist, are $0$ and $-6$. The corresponding rational value for (\ref{eq:g}) is $2$  

\subsection{Lines for $\Gamma_1(6)$-structure}\label{subs:gamma16}
Let $Z$ be a cyclic subgroup of order $6$ with generator $p$. There are $6$ elements in $T$:
\[
 (p,p,4p),\; (p,2p,3p),\; (2p,2p,2p),\; (2p,5p,5p),\; (3p,4p,5p),\; (4p,4p,4p).
\]
There are $6$ pairs of lines which may be parallel. We list only four of them, the other two can be obtained from the listed ones by sending $p$ to $-p$.
\[
 (l_{p,p,4p}, l_{2p,2p,2p}),\; (l_{p,p,4p}, l_{2p,5p,5p}),\; (l_{p,2p,3p}, l_{4p,4p,4p}),\; (l_{2p,2p,2p}, l_{4p,4p,4p})).
\]

Let us compute the slopes for the family (\ref{eq:g}), $p=-P-Q$. The slopes where defined with respect to equation of the form $y^2=x^3 + a x +b$. So whenever we have computed the (negative) slope of a line in coordinates $X$, $Y$ we add $1-\frac{k}2$ to obtain the slopes that we need:
\[
 s_{p,p,4p} = -1-\frac{k}2,\; s_{p,2p,3p} = -\frac{k}2,\; s_{2p,2p,2p} = 1-\frac{k}2,
\]
\[
 s_{2p,2p,5p} = 1+\frac{k}2,\; s_{3p,4p,5p} = \frac{k}2,\; s_{4p,4p,4p} = -1+\frac{k}2.
\]
The values of $k$ for which two lines become parallel are (recall that $k\notin\{-1,0,8\}$):
\begin{itemize}
 \item $k=-2$. This is the curve $20A2$ in Cremona's table. In this case $s_{p,p,4p}=s_{2p,5p,5p}$. We obtain relation $16 R_{E,\omega}(P+Q) = -11 R_{E,\omega}(P)$.
 \item $k=1$. This is the curve $14A4$. In this case $s_{p,2p,3p}=s_{4p,4p,4p}$ and $s_{3p,4p,5p}=s_{2p,2p,2p}$. We obtain relation $2 R_{E,\omega}(P+Q) = 5 R_{E,\omega}(P)$.
 \item $k=2$. This case already appeared in Section \ref{subs:lines:3}.
\end{itemize}

\subsection{Lines for $\Gamma_1(4)\cap\Gamma(2)$-structure}\label{subs:gamma14}
Let $Z$ be a group of order $8$ generated by $p$ and $r$ of orders $4$ and $2$ respectively. $T$ contains $9$ elements:
\[
(p,p,2p),\; (p,r,3p+r),\; (p,p+r,2p+r),\; (p+r,p+r,2p),\; (2p,r,2p+r),
\]
\[
(3p,3p,2p),\; (3p,r,p+r),\; (3p,3p+r,2p+r),\; (3p+r,3p+r,2p).
\]
Let $z_p=\alpha$, $z_{p+r}=\beta$. The slopes of the corresponding lines are
\[
s_{p,p,2p}=2\alpha,\; s_{p,r,3p+r}=\alpha-\beta,\; s_{p,p+r,2p+r}=\alpha+\beta,\; s_{p+r,p+r,2p}=2\beta,\; s_{2p,r,2p+r}=0,
\]
\[
s_{3p,3p,2p}=-2\alpha,\; s_{3p,r,p+r}=-\alpha+\beta,\; s_{3p,3p+r,2p+r}=-\alpha-\beta,\; s_{3p+r,3p+r,2p}=-2\beta.
\]
Studying the possibilities of some lines becoming parallel we obtain $4$ cases, namely $\alpha=\pm 3\beta$ and $\beta=\pm 3\alpha$. Without loss of generality we assume $\beta=3\alpha$ (the other cases can be obtained from this by changing the choice of generators). The relation between elliptic dilogarithms we get in this case is:
\[
3 R_{E,\omega}(p+r)=5R_{E,\omega}(p).
\]
Next we find the curve $E$ with $\beta=3\alpha$. By Proposition \ref{prop:z} (iii) we can find $x_p=-\frac43 \alpha^2$, $x_{p+r}=\frac{44}{3}\alpha^2$, $x_{2p}=\frac{20}3\alpha^2$, $x_{r}=-\frac{28}3\alpha^2$, $x_{2p+r}=\frac83\alpha^2$. In particular, 
\[
(x-x_{2p})(x-x_{r})(x-x_{2p+r}) = x^3 -\frac{208}3\alpha^4 x + \frac{4480}{27}\alpha^6 = x^3 + ax + b.
\]
Different choices of $\alpha\in\Q$ correspond to the same elliptic curve $24A1$ with minimal model $y^2=x^3-x^2-4x+4$. 

\section{Curves of conductor $14$}
There are $6$ curves of conductor $14$, denoted $14A1$--$14A6$ in Cremona's tables. They form the following system of isogenies:
\[
 \begin{CD}
14A4 @>>> 14A1 @>>> 14A3\\
 @VVV @VVV @VVV\\
14A6 @>>> 14A2 @>>>  14A5.
 \end{CD}
\]
The horizontal arrows have degree $3$, the vertical ones degree $2$, and the directions of the isogenies where chosen in such a way that pulling back the N\'eron differential on one curve gives the N\'eron differential on another curve.

Now we identify the curves which correspond to the identities (\ref{eq:conjectures1})--(\ref{eq:conjectures5}) in the following table:

\begin{tabular}{cccc}
Mahler measure & Deuring form & Name\\
$n(-1)$ & $D(5, 7)$ & $14A1$\\
$n(5)$ & $D(11, 49)$ & $14A2$\\
$g(1)$ & $D(-1, 1)$ & $14A4$\\
$g(7)$ & $D(5, 7)$ & $14A1$\\
$g(-8)$ & $D(-10, -8)$ & $14A6$
\end{tabular} 

\subsection{Isogenies of order $3$}
We begin with the values $n(-1)$ and $n(5)$. For any $k$ let us find an isogeny from the curve $E=D(k+6, k^2+3k+9)$ to another curve whose kernel is generated by $Q$ (see Section \ref{subs:case_n}). We send a point on $E$ to the point $(X', Y')=(-yz, y^3)$. Then
\[
 Y'-\frac{X'^3}{Y'} + 1 = -k X',
\]
so we see that $(X',Y')$ lies on $E'=D(k,1)$. Thus we obtain an isogeny $\rho_k:D(k+6, k^2+3k+9)\rightarrow D(k,1)$. We have $\rho(P_E)=P_{E'}$ and $\rho^* \omega_{E'}=3 \omega_E$. Hence
\[
 R_{E',\omega_{E'}}(P_{E'})=3 R_{E,\omega_E}([P]+[P+Q]+[P-Q]).
\]
Let $k\in\R$. Since the kernel of the isogeny does not contain real points, it preserves the real cycle. Therefore the pullback of $du$ on $E'$ is again $du$ on $E$ and 
\[
 R_{E'}(P_{E'}) = R_{E}([P]+[P+Q]+[P-Q]).
\]

Thus we obtain
\[
 n(-1) = \frac{9}{\pi} R_{E_n(-1)}([P]+[P+Q]+[P-Q]) = \frac{9}{\pi} R_{E_g(1)}(P),
\]
\[
 n(5) = -\frac{9}{2\pi} R_{E_n(5)}([P]+[P+Q]+[P-Q]) = -\frac{9}{2\pi} R_{D(5,1)}(P) = \frac{9}{2\pi} R_{E_g(-8)}(P).
\]
Note the sign change in the last equality. This is because $E_g(-8)=D(-10,-8)$ is isomorphic to $D(5,1)$ but the orientation on the set of real points is opposite. 

\subsection{Isogenies of order $2$}
Let us now find an isogeny from $E=E_g(k)=D(k-2,k)$ (as in Section \ref{subs:case_g}) to another curve whose kernel is generated by $Q$ (now $Q$ has order $2$). We send a point $(X,Y)$ on $E$ to the point $(X', Y')=(\frac{X(X-k)}{X+1}, \frac{(Y+kX)(Y+X^2)}{(X+1)^2})$. Then $(X',Y')$ lies on the curve $D(-k-4,-k^2)$, which is isomorphic to $D(-\frac{8}k-2,-\frac{8}k)=E_g(-\frac8k)$. For $k\in\R$ the map $E_g(k)\rightarrow D(-k-4,-k^2)$ does not change the orientation, while the isomorphism $D(-k-4,-k^2)\rightarrow E_g(-\frac8k)$ changes the orientation when $k<0$. 

The curves $E_g(1)$ and $E_g(-8)$ are connected by such isogenies. Let $\rho_1:E_g(1)\rightarrow E_g(-8)$ and $\rho_2:E_g(-8)\rightarrow E_g(1)$. $\rho_1$ does not change the orientation. The set $E_g(1)(\R)$ has one connected component ($\tau\in \frac12+\I\R$). Therefore $\rho_1$ restricted to this set is a double covering which means that $\rho_1^* du = 2 du$. The set $E_g(-8)(\R)$ has two connected components ($\tau\in \I\R$). Therefore $\rho_2$ is an (orientation-reversing) isomorphism on each connected component and $\rho_2^* du = - du$. We obtain
\[
 R_{E_g(-8)}(P) = 2 R_{E_g(1)}([P]+[P+Q]), \qquad R_{E_g(1)}(P) = - R_{E_g(-8)}([P]+[P+Q]).
\]
Using these equalities we express all the Mahler measures we need except $g(7)$ in terms of dilogarithms on $E_g(1)$ (for $n(-1)$ see above):
\[
 g(-8) = \frac{3}{\pi} R_{E_g(-8)}([P]-[P+Q]) = \frac{3}{\pi}R_{E_g(1)}(5[P]+4[P+Q]),
\]
\[
 g(1) = \frac{3}{\pi} R_{E_g(1)}(-[P]+[P+Q]),
\]
\[
 n(5) = \frac{9}{\pi} R_{E_g(1)}([P]+[P+Q]).
\]

Now we use the relation $2R_{E_g(1)}(P+Q) = 5R_{E_g(1)}(P)$, which was obtained in Section \ref{subs:gamma16}, to eliminate $R_{E_g(1)}(P+Q)$ and list what we have obtained so far:
\begin{align*}
 g(-8) &= \frac{45}{\pi} R_{E_g(1)}(P),\qquad &g(1) &= \frac{9}{2\pi} R_{E_g(1)}(P), \\
 n(5) &= \frac{63}{2\pi} R_{E_g(1)}(P), \qquad &n(-1) &= \frac{9}{\pi} R_{E_g(1)}(P).
\end{align*}
So the identities (\ref{eq:conjectures1}-\ref{eq:conjectures3}) and (\ref{eq:conjectures5}) are proved if we prove the following identity:
\[
 R_{E_g(1)}(P) = \frac{7}{9\pi}L(f,2).
\]

\subsection{Using Beilinson's theorem}
In this section $E=D(5,7)$.
The curve that is isomorphic to $X_0(14)$ is the curve $14A1$. It is given by the equation 
\[
 Y^2 + Y X + Y = X^3 + 4 X - 6.
\]
we make the change of variables $X'=X-2$, $Y'=Y-2X+2$ and obtain $D(5,7)$
\begin{equation}\label{eq:14a1}
 Y^2 + 5 X Y + 7 Y = X^3.
\end{equation}
This coincides with $E_n(-1)$ and $E_g(7)$. We will work with $E_g(7)$. We have the following points on the curve: $P=(0,0)$, $Q=(-1,-1)$, $-P=(0, -7)$, $P+Q=(7,-49)$, $-P+Q=(7,7)$.

The modular parametrization of (\ref{eq:14a1}) can be computed by PARI and is given by series which begin as
\begin{align*}
 X&=q^{-2}+q^{-1}+2 q + 2 q^2 + 3 q^3 +4 q^4-2q^5 + \cdots,\\
 Y&=-q^{-3} - 4q^{-2} - 6q^{-1} - 8 - 13q - 12q^2 - 14q^3 - 20q^4 +\cdots.
\end{align*}
The expansion of the corresponding newform of weight $2$ begins with
\[
 f = q - q^2 - 2 q^3 + q^4 + 2 q^6 + q^7 +\cdots.
\]

Let $E_2$ be the Eisenstein series of weight $2$ for $PSL_2(\Z)$:
\[
 E_2 =  1 - 24q - 72q^2 - 96q^3 - 168q^4 - 144q^5 - 288q^6 - 192q^7+\cdots.
\]
We verify that
\begin{align*}
 \frac{d}{2\pi\I d\tau} \log (X+1) &= \frac{1}{24}(-E_2(\tau)+2E_2(2\tau)+49 E_2(7\tau)-98 E_2(14\tau)),\\
 \frac{d}{2\pi\I d\tau} \log Y &= \frac{1}{24}(-4E_2(\tau)+16E_2(2\tau)+28 E_2(7\tau)-112 E_2(14\tau)),\\
 \frac{d}{2\pi\I d\tau} \log (Y+7X) &= \frac{1}{24}(3 E_2(\tau)+2E_2(2\tau)+21 E_2(7\tau)-98 E_2(14\tau)),
\end{align*}
which shows that the images of the cusps are the points $0$, $P$, $Q$ and $P+Q$ (because 
\begin{multline*}
(X+1)=2[Q]-2[0],\; (Y)=3[P]-3[0], (Y+7X)=2[P+Q]+[P]-3[0].)
\end{multline*}

Consider the involutions $w_2$ and $w_7$ (see Section \ref{sec:beilinson}). We have $a(2) = -1$, $a(7)=1$, therefore $\gamma(w_2) = 1$, $\gamma(w_7) = -1$. Since $\gamma(w_2)$ preserves the holomorphic differential it must act as the shift by $Q$. Therefore $w_2(0) = Q$, $w_2(P) = P+Q$. Then $w_7$ is either $x\rightarrow P-x$ or $x\rightarrow P+Q-x$. To find out which one it is we look at the identity
\[
 \frac{d}{2\pi\I d\tau} \log Y = E_2|_2(\frac{1}{6}(-1 + 2 w_2 + w_7 - 2 w_2 w_7)).
\]
This implies that ($d$ was defined in Section \ref{sec:beilinson}):
\begin{multline*}
 (Y) = d\frac16(-1 + 2 w_2 + w_7 - 2 w_2 w_7)[0] =-\frac{1}{6}(1+2w_2)(1+7w_7)(1-2 w_2)(1-w_7) [0] \\= -3(1-w_7)[0]. 
\end{multline*}
From here we see that $w_7(0) = P$, $w_7: x\rightarrow P-x$.

Now we may apply (\ref{eq:beilinson:3}) for $\alpha = 3 w_7 - 3$ and $\beta = 2 w_2 - 2$. We get 
\[
\alpha' = -\frac{(1-2w_2)(1-w_7)}6,\qquad \beta' = -\frac{(1-w_2)(1-7w_7)}{24},
\]
\[
 \gamma^*(\beta') = -\frac{(1-w_2)(1+7w_7)}{24},\qquad \alpha' \gamma^*(\beta') = -\frac{(1-w_2)(1-w_7)}{8}.
\]
This gives $\varepsilon(w_{14} \alpha' \gamma^*(\beta')) = -\frac18$, so
\[
 \langle r_E(\{F_\alpha,F_\beta\}), \omega \rangle = \frac{252}{\pi} L(f,1) L(f,2).
\]
Now we use $L(f,1) = \frac{\Omega_\R}{6}$. Since $\omega = \Omega_\R du$ 
\[
 \langle r_E(\{F_\alpha,F_\beta\}), du \rangle = \frac{42}{\pi} L(f,2).
\]

Expressing the left hand side in terms of elliptic dilogarithms, since $(F_\alpha) = 3 ([P] - [0])$, $(F_\beta) = 2([Q]-[0])$, $(F_\alpha)\ast(F_\beta)^- = 6 ([P+Q] - [P])$:
\[
 R_{E_g(7)}([P+Q]-[P]) = \frac{7}{\pi} L(f,2).
\]

This proves (\ref{eq:conjectures4}). Indeed, by (\ref{eq:deninger_g}) 
\[
 g(7) = -\frac{3}{\pi} R_{E_g(7)}([P] - [P+Q]).
\]
thus $g(7) = \frac{21}{\pi^2} L(f,2)$.

\subsection{Depth-first search}
Finally we need to relate the values $R_{E_g(7)}([P] - [P+Q])$ and $R_{E_g(1)}(P)$. First we construct an isogeny of order $3$ from $E_g(1)$ to $E_g(7)$. One can verify that the map
\[
 \rho(X,Y) = (X-2-X^{-1}+X^{-2}, Y(1 + X^{-2} - 2 X^{-3}) - 3 X + 2 + 2 X^{-1} - X^{-2} - X^{-3})
\]
is the required isogeny. Next we find $\rho^{-1} P$ and $\rho^{-1} (P+Q)$. The kernel of $\rho$ is generated by $P$. Note that $\rho(Q) = Q$. Let $\xi$ be a root of the polynomial $t^3-2 t^2 - t + 1$. Then $A=(\xi,\xi^2-1)$ is one of the preimages of $P$. One can check that $A$ has order $9$ and $3A = P$. The other preimages of $P$ are $4A$ and $7A$. Note also that $\rho$ on the set of real points preserves the orientation and is a $3$-fold covering, so $\rho^* du = 3 du$. Thus we obtain:
\[
 R_{E_g(7)}([P+Q] - [P]) = 3 R_{E_g(1)}(-[A]-[4A]-[7A]+[A+Q]+[4A+Q]+[7A+Q]).
\]

Therefore to complete the proof of (\ref{eq:conjectures1}-\ref{eq:conjectures3}) and (\ref{eq:conjectures5}) we need to prove the following:
\[
 3 R_{E_g(1)}(P) = R_{E_g(1)}(-[A]-[4A]-[7A]+[A+Q]+[4A+Q]+[7A+Q]),
\]
or, in a slightly different form,
\[
 R_{E_g(1)}([A] - [2A] + 3[3A] + [4A] - [A+Q] + [2A+Q] - [4A+Q]) = 0.
\]

We will search for parallel lines as in Section \ref{sec:search}. As a finite subgroup $Z$ we choose the subgroup generated by $A$ and all $2$-torsion points. It appears that all points of $Z$ are defined over the field generated by $7$-th roots of unity. So let $\gamma$ be a primitive $7$-th root of unity. Then we choose $\xi=\gamma+\gamma^{-1}+1$. The $X$-coordinates of the $2$-torsion points different from $Q$ are given by $\frac{3\pm\sqrt{-7}}8$. Let $Q'$, $Q''$ be the points corresponding to plus nd minus signs respectively. We choose $1+2\gamma+2\gamma^2+2\gamma^4$ as the square root of $-7$. This information is enough to determine the coordinates of all points from $Z$.

% Next we find the values of $z_p$ for $p\in Z$. We do it using the following observations (some of the values $s_{p,q,r}$ where computed in Section \ref{sec:search}):
% \begin{itemize}
%  \item For all points of order $2$ $z_p=0$.
%  \item Since $s_{3A,3A,3A}=\frac12$, $z_{3A} = \frac16$.
%  \item We must have $z_A+z_{2A}+z_{-3A} = s_{A,2A,-3A}$ and $z_A+z_A+z_{-2A} = s_{A,A,-2A}$. Therefore $3 z_A = z_{3A} + s_{A,2A,-3A} + s_{A,A,-2A}$. From this $z_A = \frac{1}{18} + \frac{1}{3}(\gamma^2+2\gamma^3+2\gamma^4+\gamma^5)$.
%  \item If $q$ has order $2$ and $p$ is arbitrary $z_p+z_q+z_{-p-q} = s_{p,q,-p-q}$. From this $z_{p+q} = z_p - s_{p,q,-p-q}$. In particular, 
% \[
% z_{A+Q} = \tfrac{14}9 + \tfrac{2}{3}(2\gamma^2+\gamma^3+\gamma^4+2\gamma^5),\quad z_{A+Q'} = -\tfrac49 -\tfrac16(4\gamma^5 + 5\gamma^4 + 2\gamma^3 + \gamma^2 +3\gamma), 
% \]
% \[
% z_{3A+Q} = \tfrac23,\quad z_{3A+Q'} = -\tfrac13 +\tfrac12(\gamma^4 + \gamma^2 + \gamma). 
% \]
% \end{itemize}
% Applying elements of the Galois group we obtain all other values.
% 
Using a computer search we found several pairs of parallel lines. Take the line passing through $A$, $A+Q'$, $-2A+Q'$ and the line passing through $2A$, $3A+Q''$, $4A+Q''$. Both lines have (negative) slope $-\gamma-1$. Therefore $R_E$ vanishes on the following divisor:
\begin{multline*}
 D_1 = -4[A]+3[2A]-[2A+Q]+[4A+Q]-4[A+Q']+3[2A+Q']-[4A+Q']-\\ [2A+Q''] + 2[3A+Q''] + 3[4A+Q''].
\end{multline*}
If we apply an element of the Galois group of $\Q(\gamma)/\Q$ to $D_1$ we also obtain a divisor in the kernel of $R_E$. Therefore we may take the trace:
\[
 \Tr D_1 = -7 \Tr [A] + 2 \Tr[A+Q] - 4 \Tr[A+Q'] + 2 \Tr[3A+Q'] \in \kernel R_E.
\]
Using the multiplication by $2$ endomorphism of $E$ we have the following relation for any $p\in E(\C)$:
\[
 R_E(-[2p]+2[p]+2[p+Q]+2[p+Q']+2[p+Q''])=0.
\]
Putting $p=A$ and taking trace we obtain:
\[
 3\Tr[A] + 2\Tr[A+Q] + 4 \Tr[A+Q'] \in \kernel R_E. 
\]
For $p=3A$, similarly:
\[
 3\Tr[3A] + 2\Tr[3A+Q] + 4 \Tr[3A+Q'] \in \kernel R_E.
\]
Therefore we may eliminate $\Tr[A+Q']$ and $\Tr[3A+Q']$ from $\Tr D_1$:
\[
 -8\Tr[A] + 8\Tr[A+Q] - 3\Tr[3A] - 2 \Tr[3A+Q] \in \kernel R_E.
\]
This means
\[
 -16([A]+[4A]+[7A]) + 16 ([A+Q]+[4A+Q]+[7A+Q]) - 18 [3A] - 12[3A+Q] \in \kernel R_E.
\]
Recall that $R_E(3A+Q) = \frac52 R_E(3A)$. We obtain
\[
 -([A]+[4A]+[7A]) + [A+Q]+[4A+Q]+[7A+Q] - 3 [3A] \in \kernel R_E.
\]
This completes the proof.

Another pair of parallel lines is $(A,A,-2A)$ and $(-4A+Q,2A+Q',2A+Q'')$, but taking trace of the corresponding divisor does not produce any non-trivial relation.

% The bibliography was generated with BibTeX:
% \bibliography{../commons/refs2}

\end{document}